\documentclass[reqno,a4paper,12pt]{amsart}
\usepackage[hmargin=3cm]{geometry}
\usepackage{graphicx}

\usepackage{amscd}
\usepackage{latexsym}
\usepackage{amsmath}
\usepackage{amsfonts}
\usepackage{amssymb}
\usepackage{amsthm}
\usepackage{mathrsfs}
\usepackage{comment}

\usepackage[pdftex]{hyperref}
\usepackage[titletoc]{appendix}
\usepackage{bookmark}
\bookmarksetup{numbered}

\usepackage{xfrac}
\usepackage{tikz}
\usetikzlibrary{matrix,arrows,decorations.pathmorphing}
\usepackage{graphicx}
\usepackage{mathtools}

\usepackage[utf8x]{inputenc}
\usepackage[english]{babel}

\newtheorem{theorem}{Theorem}[section]
\newtheorem{proposition}[theorem]{Proposition}
\newtheorem{lemma}[theorem]{Lemma}
\newtheorem{corollary}[theorem]{Corollary}
\newtheorem*{theorem*}{Theorem}

\theoremstyle{remark}
\newtheorem{example}[theorem]{Example}

\theoremstyle{definition}
\newtheorem{definition}[theorem]{Definition}

\theoremstyle{remark}
\newtheorem{remark}[theorem]{Remark}

\DeclareFontFamily{OT1}{rsfs}{}
\DeclareFontShape{OT1}{rsfs}{n}{it}{<-> rsfs10}{}
\DeclareMathAlphabet{\curly}{OT1}{rsfs}{n}{it}

\DeclareMathOperator\im{im}

\DeclareMathOperator\rk{rk}

\newcommand{\Q}{\mathrm{Quot}}
\newcommand{\Hilb}{\mathrm{Hilb}}

\newcommand{\C}{\mathbb{C}}

\newcommand{\proj}{\mathbb{P}^2}

\newcommand{\linf}{\ell_\infty}

\newcommand{\GL}{\operatorname{GL}}

\newcommand{\Oo}{\curly O}
\newcommand{\Or}{\curly O^{\oplus r}}
\newcommand{\End}{\operatorname{End}}
\newcommand{\Hom}{\operatorname{Hom}}
\newcommand{\Supp}{\operatorname{Supp}}
\newcommand{\p}{\pi^{-1}(n[0])}

\newcommand{\dual}{\check{\hspace{1mm}}}

\DeclareMathOperator\Sym{Sym}

\DeclareMathOperator\ch{ch}

\let\ra\longrightarrow
\let\xra\xrightarrow
\def\iso{\cong}

\begin{document}
\bibliographystyle{plain}

\title{Topology of moduli spaces of framed sheaves}
\author{Gharchia Abdellaoui}
\address{Institut f{\"u}r Mathematik, University of Osnabr{\"u}ck, Albrechtstraße 28a
\\49076 Osnabr{\"u}ck,
Germany}
\email{ghabdellaoui@uos.de}

\begin{abstract}

In this paper we show that the moduli space of framed torsion-free sheaves on a certain class of toric surfaces admits a filtrable Bia{\l}ynicki-Birula decomposition determined by the torus action. The irreducibility of this moduli space follows immediately as a corollary. Moreover, using its filtrable decomposition we show that the moduli space shares the same homotopy type of an invariant proper compact subvariety having the same fixed points set.

We start our study from the moduli space of framed torsion-free sheaves on the projective plane. Afterwards, we generalize the results to toric surfaces under a few assumptions.

\end{abstract}

\maketitle

\tableofcontents

\section{Introduction}

The decomposition of algebraic varieties determined by a torus action was introduced in \cite{BB} for non-singular complete varieties acted on algebraically by a torus with a non-empty fixed points set. These decompositions are often referred to as the Bia{\l}ynicki-Birula decompositions. In \cite{konarski} Konarski generalized the results of Bia{\l}ynicki-Birula to the case of a complete normal variety and in addition he addressed the case of non-normal varieties.

Bia{\l}ynicki-Birula's decompositions play an important role in the study of topological properties of moduli spaces acted on by a torus. In the present paper we study some of these properties for moduli spaces of framed torsion-free sheaves.

In section \ref{generalities} we review some results on the moduli space $M(r,n)$ studied in \cite{Nak,NY,NakYosh}. This moduli space admits a projective morphism to the moduli space $M_0(r,n)$ of ideal instantons on $S^4.$ Both $M(r,n)$ and $M_0(r,n)$ admit a torus action under which the projective morphism from $M(r,n)$ to $M_0(r,n)$ is equivariant. It is possible then to define a subvariety $\p$ of $M(r,n)$ which is invariant under the torus action. We prove this subvariety is irreducible using the fact that it is isomorphic to  punctual quot scheme.

In section \ref{Decomposition of a variety determined by an action of a torus} we will review some results  of Bia{\l}ynicki-Birula \cite{BB,BB'} and translate them to the case of quasi-projective varieties under the assumption that the limits exist (Section \ref{not complete}).

Following the results of section \ref{Decomposition of a variety determined by an action of a torus}, we show in section \ref{decomp M and pi} that $M(r,n)$ admits a Bia{\l}ynicki-Birula plus-decomposition and the union of the minus-cells builds up $\p.$ Moreover, these decompositions are filtrable.

In section \ref{more topology} we show that the inclusion of $\p$ into $M(r,n)$ induces isomorphisms of homology groups with integer coefficients. Furthermore, we show that this inclusion induces a homotopy equivalence between $M(r,n)$ and $\p,$ see Theorem \ref{homotopy equivalence}.

In \cite[Theorem 3.5(2)]{NY}, Nakajima and Yoshioka stated that $M(r,n)$ is homotopy equivalent to the proper subvariety $\p$ but the proof is rather obscure since it refers to papers which seem not to contain the claimed arguments.
We then find interesting to give a detailed proof of the homotopy equivalence.

In section \ref{gener} we generalize this study to the moduli space $M(S)$ of framed sheaves on a toric surface $S$ where we assume that there exists a projective morphism of toric surfaces $p:S \to \mathbb P^2$ of degree 1. We show that $M(S)$ admits a Bia{\l}ynicki-Birula plus-decomposition and the union of the minus-cells builds up a proper compact subvariety $\widetilde N \subset M(S).$ Using this result we prove that M(S) is homotopy equivalent to the compact subvariety $\widetilde N,$ see Theorem \ref{homotopy equivalence M(S)}.

\subsection*{Acknowledgements}This paper covers a part of my PhD thesis. I am grateful to my advisor Ugo Bruzzo for proposing the problem and for his guidance through this work. I would also like to thank Michel Brion, Tam\'as Hausel and Hiraku Nakajima for brief but very useful discussions. Special thanks to Marcos Jardim for very useful comments. Finally, I warmly thank the Institut für Mathematik at the University of Osnabrück for the friendly atmosphere while revising this paper.

\section{Generalities on the moduli space of framed torsion-free sheaves on the projective plane}\label{generalities}
Most of the material in this section can be found in \cite{Nak,NY,NakYosh}. Let $M(r,n)$ be the  moduli space of framed torsion free sheaves on $\proj$ with rank $r$ and second Chern class $n$ parametrizing isomorphism classes of $(\curly E,\phi)$ such that

\begin{enumerate}
  \item $\curly E$  is a torsion free sheaf on $\proj$  of rank $ r$ and second Chern class $n.$ $\curly E$ is locally free in a neighborhood of $\linf$.
  \item $\phi \colon \curly E|_{\linf} {\overset{\cong}{\longrightarrow}}\Or_{\linf}$ \emph{framing at infinity},
\end{enumerate}
where $\linf=\{[0:z_1:z_2] \in\ \proj \}$ is the line at infinity. We say two framed sheaves $(\curly E,\phi),(\curly E',\phi')$ are \emph{isomorphic} if there exists an isomorphism $\psi :\curly E \rightarrow \curly E'$ such that the following diagram commutes
\begin{center}
\begin{tikzpicture}
\matrix(m)[matrix of math nodes, row sep=2em, column sep=1.5em, text height=1.5ex, text depth=0.25ex]
{\curly E|_{\linf}& &\curly E'|_{\linf}\\
&\Or_{\linf}& \\};
\path[->,font=\scriptsize]
(m-1-1) edge node[auto] {$\psi|_{\linf}$} (m-1-3)
(m-1-1) edge node[auto]{}
        node[below] {$\phi$}(m-2-2)
(m-1-3) edge node[auto] {$\phi'$} (m-2-2);
\end{tikzpicture}
\end{center}

$M(r,n)$ is known to be a nonsingular quasi-projective variety of dimension $2nr$.

Let $M_0^{\operatorname{reg}}(r,n)\subset M(r,n)$ be the open subset of locally free sheaves. We define the Uhlenbeck (partial) compactification of $M_0^{\operatorname{reg}}$ set theoretically as follows
$$M_0(r,n) :=
   \bigsqcup_{k=0}^n M_0^{\operatorname{reg}}(r,n-k)\times \Sym^k\C^2,$$
where $\Sym^k\C^2$ is the $k$-th symmetric product of $\C^2$.

We recall that there exists a description of $M(r,n)$ by a quotient of linear data $(B_1,B_2,i,j)$ by the action of $\GL_n(\C)$ such that
\begin{align}
(i)\:\: &[B_1,B_2]+ij=0, \label{cond 1}
\\(ii)\:\: &\text{there exists no subspace } S \subsetneq \C^n \text{ such that } B_\alpha(S) \subset S  \text{ (} \alpha=1,2 \text{) } \label{cond 2}\\ \notag &\text{and } \im i\subset S,
\end{align}
where $B_1,B_2 \in \End(\C^n)$, $i \in \Hom(\C^r,\C^n)$, $j \in \Hom(\C^n,\C^r)$ and for $g \in\ \GL_n(\C)$ the action is given by
$$g \cdot (B_1,B_2,i,j)=( g B_1 g^{-1}, g B_2 g^{-1}, gi, jg^{-1}).$$

$M_0(r,n)$ can be endowed with a scheme structure by the following description
\begin{equation} \label{ADHM for M0}
M_0(r,n)\cong \left.\left\{ (B_1,B_2,i,j) \mid [B_1,B_2]+ij=0 \right\}\right. /\!\!/
 \GL_n(\C),
\end{equation}
where $/\!\!/$ denotes the GIT quotient and the open locus $M_0^{\operatorname{reg}}(r,n)$ consists of the closed orbits such that the stabilizer is trivial.

The moduli spaces $M(r,n)$ and $M_0(r,n)$ are related by the following projective morphism
\begin{align}\label{projective morphism}
\pi\colon \hspace{0.3cm} M(r,n) &\to M_0(r,n) \\ \notag
(\curly E,\phi) \hspace{.15cm} &\mapsto
   ((\curly E{\dual\dual},\phi), \operatorname{Supp}(\curly E{\dual\dual}/\curly E))\in
   M_0^{\operatorname{reg}}(r,n-k)\times \Sym^{k}\C^2 ,
\end{align}
where $\curly E{\dual\dual}$ is the double dual of $\curly E$ and
$\Supp(\curly E{\dual\dual}/\curly E)$ is the topological support of
$\curly E{\dual\dual}/\curly E$ counted with multiplicities.

For $k=n$, $M_0^{\operatorname{reg}}(r,0)\times \Sym^{n}\C^2 \simeq \Sym^{n}\C^2$ since $M_0^{\operatorname{reg}}(r,0)$ consists of one point that is the isomorphism class of the pair $\Or_{\proj}$ together with the trivial framing.

Consider the point $n[0] \in \Sym^{n}\C^2$ that is the point $0 \in \C^2$ counted $n$ times. The inverse image of $n[0]$ by $\pi$ is defined as follows $$\p= \{(\curly E,\phi) \mid \curly E{\dual\dual}\simeq \Or_{\proj},\hspace{0.2cm} \Supp(\Or_{\proj}/\curly E)=n[0] \}/ \iso .$$
$\p$ is a compact subvariety of $M(r,n)$ since $\pi$ is proper.

\begin{theorem}[{\cite[Theorem 3.5 (1)]{NY}}] \label{iso quot sch}
$\p$  is isomorphic to the punctual quot-scheme $\Q(r,n)$ parameterizing zero dimensional quotients $\Or_{\proj} \to \curly Q$ with $\Supp (\curly Q) = n[0]$.
\end{theorem}

\begin{proof}
Given a point in
$$\pi^{-1}(n[0])= \left\{ (\curly E,\phi) \hspace{0.2cm} \left| \hspace{0.2cm} \curly E{\dual\dual}\simeq \Or_{\proj},\hspace{0.2cm} \Supp(\Or_{\proj}/\curly E)=n[0] \right\} \right/ \iso $$
one has the exact sequence
$$0\rightarrow \curly E \rightarrow \Or_{\proj} \rightarrow \Or_{\proj}/\curly E\rightarrow 0.$$
Thus the quotient $\Or_{\proj} \twoheadrightarrow \Or_{\proj}/\curly E$ is a point in the punctual quot-scheme defined by $$\Q(r,n):= \left\{\Or_{\proj} \twoheadrightarrow \curly Q \left| \Supp \curly Q= n[0]\right\} \right/ \iso~.$$

Conversely, given a point in $\Q(r,n),$ let $\curly K :=\ker(\Or_{\proj} \twoheadrightarrow \curly Q$ in $\Q(r,n)).$ We have the exact sequence
$$0\rightarrow \curly K \rightarrow \Or_{\proj} \rightarrow \curly Q \rightarrow 0.$$
$\curly K$ is torsion free since it injects into a locally free sheaf. By using Lemma \ref{lemma}, it follows that $\curly K{\dual\dual}\simeq \Or_{\proj}.$
Notice that $\curly Q$ is supported on the point $0 \in \proj$ so it vanishes on $\proj\setminus\{0\}.$ In particular, $\curly Q|_{\linf}$ vanishes and the exact sequence reduces to $$0\ra \curly E|_{\linf} \xra{\hspace{1mm} \cong \hspace{1mm}} \Or_{\linf} \ra 0.$$
This isomorphism defines the framing $\phi.$
The pair $(\curly E, \phi)$ is a point in $\p.$
\end{proof}

\begin{theorem} \label{p irred proj}
$\p$ is irreducible projective of dimension $n(r+1).$
\end{theorem}

\proof
The quot scheme is projective \cite[Theorem 2.2.4]{HuybLehn} and irreducible of dimension $n(r+1)$ \cite[Theorem 6.A.1]{HuybLehn}. Hence by Theorem \ref{iso quot sch}, $\p$ is an irreducible projective variety of dimension $n(r+1)$ as well.
\endproof

\subsection{The Torus action} We will follow the same notations as in \cite{NY}. Let $\widetilde{T}:=\C^* \times \C^* \times T$ where $T$ is the maximal torus in $\GL(r,\C).$ The action of $\widetilde{T}$ on $M(r,n)$ is defined as follows.\\
For $(t_1,t_2)\in \C^*\times\C^*$ let $F_{t_1,t_2}$ be the automorphism of $\mathbb P^2$ defined by $$F_{t_1,t_2} ([z_0: z_1 : z_2]):=[z_0: t_1 z_1 : t_2 z_2]$$
and for $(e_1,\dots,e_r) \in T$ let $G_{e_1,\dots,e_r}$ be the isomorphism of $\Oo_{\linf}^{\oplus r}$ defined by $$G_{e_1,\dots,e_r}(s_1,\dots, s_r):=(e_1 s_1, \dots, e_r s_r).$$
Then the action of $\widetilde{T}$ on a pair $(\curly E,\Phi)\in M(r,n)$ is defined by
$$(t_1,t_2,e_1,\dots,e_r)\cdot (\curly E,\Phi):= \big((F_{t_1,t_2}^{-1})^* \curly E, \Phi'\big)$$ where $\Phi'$ is given by: $$(F_{t_1,t_2}^{-1})^* \curly E|_{\linf}
   \xrightarrow{(F_{t_1,t_2}^{-1})^*\Phi}
   (F_{t_1,t_2}^{-1})^* \Oo_{\linf}^{\oplus r}
   \longrightarrow \Oo_{\linf}^{\oplus r}
   \xrightarrow{G_{e_1,\dots, e_r}} \Oo_{\linf}^{\oplus r}$$

The $\widetilde{T}$ action is defined in a similar way on $M_0(r,n)$ and the map $\pi$ given in \eqref{projective morphism} is equivariant.

One can define the torus action on $M(r,n)$ and $M_0(r,n)$ using the description by the linear data defined in section \ref{generalities}. This action is given by
$$(B_1,B_2,i,j)=(t_1B_1,t_2B_2,ie^{-1},t_1t_2ej)$$ where $t_1,t_2 \in \C^*$ and $e=(e_1,e_2,\dots,e_r)\in \C^r.$ This action preserves the conditions \eqref{cond 1}, \eqref{cond 2}, and commutes with the action of $\GL(r,\C).$

\begin{theorem}[{\cite[Proposition 2.9]{NakYosh}}]\label{fixed points}
\hspace{-1cm}\begin{enumerate}
\item The fixed points set $M(r, n)^T$ consists of finitely many points.
\item The fixed points set $M_0(r, n)^T$ consists of a single point $n[0] \in Sym^n\C^2.$
\end{enumerate}
\end{theorem}
\begin{remark}
From Theorem \ref{fixed points} it follows that $\pi^{-1} (n[0])$ contains all the fixed points of $M(r,n)$ since $\pi$ is equivariant.
\end{remark}

\section{Decomposition of a variety determined by an action of a torus} \label{Decomposition of a variety determined by an action of a torus}

Let $X$ be a non-singular algebraic variety (not necessarily complete) over $\C$ and let $\C^*$ be the multiplicative group. Assume $\C^*$ acts algebraically on $X$ with a non-empty fixed points set $X^{\C^*}.$ Let $F_1, \dots , F_s$ be the connected components of $X^{\C^*}.$
For any $x \in X,$ one has the orbit morphism
\begin{align}\label{orbit morphism}
\phi_x \colon \hspace{0.3cm} \C^* &\to X \\ \notag
t \hspace{.15cm} &\mapsto
   t \cdot x.
\end{align}

\subsection{Case of complete varieties} When $X$ is complete the morphism \eqref{orbit morphism} extends to
\begin{align*}
\overline{\phi}_x \colon \hspace{0.3cm} \mathbb P^1 &\to X \\
t \hspace{.15cm} &\mapsto
   t \cdot x,
\end{align*}
defining the limits $\overline{\phi}_x(0)=\lim_{t \to 0} t\cdot x$ and $\overline{\phi}_x(\infty)=\lim_{t \to \infty} t\cdot x$, see \cite{BB'}.

Let us define the following subsets of $X$
$$X^+_i:=\{x\in X \mid \lim_{t \to 0} t\cdot x \in F_i \}, \: i=1,\dots, s.$$
$$X^-_i:=\{x\in X \mid \lim_{t \to \infty} t\cdot x \in F_i \}, \: i=1,\dots, s.$$
These form a decomposition of $X$ into subspaces $X= \bigcup_i X_i^+=\bigcup_i X_i^-,$ the so called \emph{plus} and \emph{minus-decompositions} \cite{BB'}. The subsets $X _i^+$ and $X_i^-$ are locally closed by \cite[Theorem 4.1]{BB}.

Now assume the fixed points set is finite. Then the Bia{\l}ynicki-Birula theorem \cite[Theorem 4.3]{BB} can be stated as follows

\begin{theorem}\label{BB finite}
Let $X$ be a complete variety and let the fixed points set be $X^{\C^*}= \{x_1, \ldots, x_s\}$. For any $i= 1, \ldots, s$ there exists a unique $\C^*$-invariant decomposition of $X,$ $X= \bigcup_{i=1}^s X_i^+$ $($resp.  $X= \bigcup_{i=1}^s X_i^-)$ such that
\begin{enumerate}
   \item $x_i \in X^+_i  (\text{resp. } x_i \in X^-_i),$
   \item $X^+_i (\text{resp. }X_i^-)$ is isomorphic to an affine scheme,
   \item for any $x_i$, $T_{x_i}(X^+_i) = T_{x_i}(X)^+ (\text{resp. }T_{x_i}(X^-_i) = T_{x_i}(X)^-).$
 \end{enumerate}
\end{theorem}

\definition \label{def fitrable}
A decomposition $\{X^+_i\}$ (resp. $\{X^-_i\}$) is said to be \emph{filtrable} if there is a decreasing sequence of closed subvarieties $$X=X_1 \supset X_2 \supset \cdots \supset X_s\supset X_{s+1}=\emptyset,$$ such that the \emph{cells} of the decomposition are $X^+_i=X_i \setminus X_{i+1}$ (resp. $X^-_i=X_i \setminus X_{i+1}$) for $i=1, \dots ,s.$
\enddefinition

\begin{remark}\label{rem}
\begin{enumerate}
  \item $X^+_i=X_i \setminus X_{i+1} \subset X_i$ and $X_i$ is closed in $X$ so the closure $\overline{X^+_i}$ of $X^+_i$ lies in $X_i=\cup_{j\geq i} X^+_j.$ Hence a filtrable decomposition implies that for each $i$
\begin{equation}\label{closure}
\overline{X^+_i}\subset \bigcup_{j\geq i} X^+_j.
\end{equation}
Notice that \eqref{closure} also holds for the minus decomposition.
  \item A filtrable decomposition yields the existence of a unique cell of maximal dimension $X^+_1$ (resp. $X^-_1$) since its complement $X_2$ is closed in $X.$
  \item Bia{\l}ynicki-Birula's decompositions of projective varieties are filtrable \cite[Theorem 3]{BB'}.
\end{enumerate}
\end{remark}

\begin{example}

\begin{enumerate}
\item Let $X=\mathbb P^l$  with the following $\C^*$-action
\begin{align*}
\omega\colon \hspace{0.5cm} \C^*\times \mathbb P^l \hspace{.2cm} &\longrightarrow \hspace{.2cm} \mathbb P^l\\
(t, (x_0:x_1: \cdots :x_l))&\longmapsto (t^l x_0: t^{l-1} x_1: \cdots :x_l),
\end{align*}
The fixed points set of this action is given by
$$({\mathbb P^l})^{\C^*}=\big\{x_0=(1:0:\cdots:0), \dots, x_l=(0:0:\cdots:1)\big\},$$
and the corresponding plus-cells are $X^+_0=\{x_0\}, X^+_1=\mathbb{A}^1, \dots, X^+_l=\mathbb{A}^l.$\\
Then the Bia{\l}ynicki-Birula plus-decomposition for $\mathbb P^l$ determined by the action $\omega$ is $\mathbb P^l = \{(1:0:\cdots:0)\}\cup \mathbb{A}^1 \cup \cdots \cup \mathbb{A}^l.$ This decomposition is filtrable. Indeed $\mathbb P^l$ has a filtration $\mathbb P^l =X_l\supset X_{l-1}\supset \cdots \supset X_0\supset X_{-1}=\emptyset$ where $X_i=X^+_i\cup X^+_{i-1}.$

Similarly, one can define the Bia{\l}ynicki-Birula minus-decomposition of $\mathbb P^l.$

\item A less elementary example is the Bia{\l}ynicki-Birula decomposition of the Hilbert scheme $\Hilb^d (\mathbb P^2)$ of $d$ points in $\mathbb P^2.$ This was done in detail in \cite{E.S}.
\end{enumerate}
\end{example}

\subsection{Case of non-complete varieties}\label{not complete}
Suppose $X$ is not complete and $\lim_{t \to 0} t\cdot x$ exists for every $x \in X$. In this case, the morphism \eqref{orbit morphism} extends to $\C$
\begin{align*}
\phi'_x \colon \hspace{0.3cm} \C &\to X \\
t \hspace{.15cm} &\mapsto
   t \cdot x,
\end{align*}
where $\phi'_x(0):= \lim_{t \to 0} t\cdot x.$
\\For $ i=1,\dots ,s$ one can define the subsets $X^+_i:=\{x\in X \mid \lim_{t \to 0} t\cdot x \in F_i \}$ which form a decomposition of $X$ into subspaces $X= \bigcup_i X_i^+$ \cite{BB'}. The cells $X _i^+$ are locally closed by \cite[Theorem 4.1]{BB}. By the same theorem \cite[Theorem 4.1]{BB} and since the variety is non-singular the cells of the decomposition are non-singular as well.

Now suppose $X$ as above is quasi-projective and the fixed points set is finite. The cells of the decomposition are isomorphic to affine spaces by Theorem \ref{BB finite}. Under these assumptions the following theorem holds.

\begin{theorem} \label{filtrable}
The Bia{\l}ynicki-Birula plus-decomposition of $X$ is filtrable.
\end{theorem}

\proof
According to \cite[Theorem 1]{sumi} there exists an equivariant projective embedding $X \to \mathbb P^l$ for some $l.$
\\Consider the action of $\C^*$ on $\mathbb P^l$ and let $(\mathbb P^l)^{\C^*}=\{x_1, \dots, x_q\}$ be the finite set of fixed points.
Let $\{P_j\}$ be the Bia{\l}ynicki-Birula plus-decomposition of $\mathbb P^l.$ Then the fixed points set of $X$ is the intersection $$X^{\C^*}=(\mathbb P^l)^{\C^*} \cap X=\{x_{\sigma(1)}, \dots, x_{\sigma(q)}\},$$
and the cells of the decomposition are $X^+_{\sigma(j)}=X\cap P_j.$
\\ Since the decomposition $\{P_j\}$ is filtrable we consider the following filtration
$$\mathbb P^l=Y_1 \supset Y_2 \supset \cdots \supset Y_q\supset Y_{q+1}=\emptyset,$$
where $P_j= Y_j\setminus Y_{j+1}$ for $j=1,\dots,q.$ Then $$X^+_{\sigma(j)}=X\cap P_j= X \cap (Y_j\setminus Y_{j+1})= (X\cap Y_j) \setminus (X\cap Y_{j+1})= X_{\sigma(j)} \setminus X_{\sigma(j+1)},$$
where $X_{\sigma(j)}:=(X\cap Y_j)$ for $j=1,\dots,q$ and $X_{\sigma(j+1)} \subset X_{\sigma(j)}.$ Notice that for each $j,$ $X_{\sigma(j)}=X\cap Y_j \subset X$ is a closed subvariety of $X$ since $Y_j$ is a closed subvariety of $\mathbb P^l.$ It follows that the plus-decomposition of $X$ is filtrable.
\endproof

\begin{remark}
Theorem \ref{filtrable} is also true in the case the fixed points set is not finite. The proof is basically the same. \\
The same holds if instead of considering the existence of the limit $\lim_{t \to 0} t\cdot x,$ one considers the existence of the limit $\lim_{t \to \infty} t\cdot x.$
\end{remark}

If we have an action of an $n$-dimensional torus the Bia{\l}ynicki-Birula decompositions enjoy the same properties as above by the following lemma.

\begin{lemma} \label{one-parameter subgroup}
Suppose $X$ is a normal algebraic variety endowed with a linear action of a torus $T.$ Suppose the $T$-action gives rise to a nontrivial set of fixed points $X^T.$ Then there exists a one-parameter subgroup $\lambda \in Y(T)$ such that $X^\lambda=X^T.$
\end{lemma}

\begin{remark} Such a one-parameter subgroup is called \emph{regular}.
\end{remark}

\proof
By \cite[Corollary 2]{sumi} $X$ can be covered by a finite number of affine $T$-invariant open subsets. So we may assume that $X$ is affine hence a closed subvariety of $\mathbb A^l$ for some $l.$ Since the torus $T$ acts linearly on $\mathbb A^l$ we will assume $X=\mathbb A^l.$ Call the weights of $T$ in $\mathbb A^l$ by $\chi_1,\dots, \chi_l,$ then it is enough to choose a one-parameter subgroup $\lambda$ such that $<\chi_i,\lambda>\neq 0$ for all $i.$
\endproof

\section{Decompositions of $M(r,n)$ and $\p$ determined by the torus action}\label{decomp M and pi}
In this section, following the results of section \ref{Decomposition of a variety determined by an action of a torus} we show that both $M(r,n)$ and $\p$ admit a filtrable Bia{\l}ynicki-Birula decomposition. By Lemma \ref{one-parameter subgroup}, it is enough to consider the action of a regular one-parameter subgroup. From now on we will consider the action of $\C^*$ on $M(r,n)$ and on $\p.$

\proposition \label{limit exists}
For every element $x \in M(r,n),$ the limit $\lim_{t\rightarrow 0} t\cdot x$ exists and lies in $M(r,n)^{\C^*}.$
\endproposition

\proof
Using the description \eqref{ADHM for M0} of $M_0(r,n)$, there exists a one-parameter subgroup of $\widetilde T$ such that for all $x $ in $ M_0(r,n),$ the limit $\lim_{t \rightarrow 0} t \cdot x$ exists and is $(B_1,B_2,i,j)=(0,0,0,0).$
The point $(B_1,B_2,i,j)=(0,0,0,0)$ is identified by the description \eqref{ADHM for M0} with the point $n[0] \in S^n\C^2$ which is the only fixed point of $M_0(r,n)$.
Since $\pi$ is a projective morphism, for all $ x $ in $M(r,n)$ the limit $\lim_{t \rightarrow 0} t \cdot x$ exists and lies in $\p.$ In particular, it is a fixed point.
\endproof

\begin{theorem} \label{BB decomp of M}
$M(r,n)$ admits a Bia{\l}ynicki-Birula plus-decomposition into affine spaces. Moreover, this decomposition is filtrable.
\end{theorem}
\proof
By Proposition \ref{limit exists}, the limit $\lim_{t \to 0} t\cdot x$ exists and is a fixed point. It follows that the orbit morphism
\begin{align*}
\phi_x \colon \hspace{0.3cm} \C^* &\to M(r,n) \\
t \hspace{.15cm} &\mapsto
   t \cdot x
\end{align*}
extends to
\begin{align*}
\phi'_x \colon \hspace{0.3cm} \C &\to M(r,n) \\
t \hspace{.15cm} &\mapsto
   t \cdot x,
\end{align*}
where $\phi'_x(0):= \lim_{t \to 0} t\cdot x.$
Hence $M(r,n)$ admits a Bia{\l}ynicki-Birula decomposition into affine spaces by Theorem \ref{BB finite}.
This decomposition is filtrable by Theorem \ref{filtrable} since $M(r,n)$ is quasi-projective.
\endproof

Let $M(r,n)^{\C^*}=\{x_i \mid i=1,\dots, m \}$ and denote the cells of the decomposition by $$ M^+_i:=M(r,n)^+_i=\{x\in M(r,n) \mid \lim_{t \to 0} t\cdot x=x_i \}.$$

The limit $\lim_{t \to \infty} t\cdot x$ does not exist for all $x \in M(r,n).$ We will only consider the points of $M(r,n)$ such that the limit exists. These points define the subspaces $$M^-_i:=M(r,n)^-_i=\{x\in M(r,n) \mid \lim_{t \to \infty} t\cdot x \text{ exists and equals } x_i \}.$$

\begin{theorem} \label{union of minus cells}
The subvariety $\p$ is the union $\p=\bigcup_i M_i^-.$ Moreover, this is a filtrable decomposition of $\p$ into affine spaces.
\end{theorem}

\proof
The proof of the first claim can be found in \cite[Theorem 3.5(3)]{NY}. In fact, for any $x\neq n[0]$ in $M_0(r,n)$ the limit $\lim_{t\to \infty}t \cdot x$ does not exist. Hence by the projective morphism \eqref{projective morphism} we deduce $\p=\bigcup_i M_i^-.$ To show this is a filtrable decomposition we apply the proof of Theorem \ref{filtrable} to the equivariant embedding $\p \hookrightarrow \mathbb P^l$ obtained by composing the equivariant embeddings $\p \hookrightarrow M(r,n)$ and $M(r,n)\hookrightarrow \mathbb P^l.$ Finally, regarding the $M^-_i$'s as subspaces of $M(r,n)$ and using \cite[Theorem 4.1(b)]{BB}, we conclude they are isomorphic to affine spaces.
\endproof

\begin{remark}
The subvariety $\p$ being isomorphic to the punctual quot scheme is not smooth. Still the Bia{\l}ynicki-Birula decompositions for $\p$ exist.
From \cite[Theorem 1]{sumi} together with \cite[Lemma 8]{sumi} it follows that for any $x\in M(r,n)$ there exists an equivariant embedding of some neighborhood of $x$ into $\mathbb P^l.$ Hence for any $x\in \p$ there exists an equivariant embedding of some neighborhood of $x$ into $\mathbb P^l$ by composition since $\p$ has an equivariant embedding into $M(r,n).$ Thus the results of \cite[\S 1]{konarski} hold for $\p.$
\end{remark}

\section{Topological properties}\label{more topology}
Let $J=\{1,2,\dots, m\}$ be the set of indices so that the fixed points set $X^{\C^*}=\{x_j \mid j\in J\}$ yield the following filtrations.
\begin{equation} \label{filtration of M}
M(r,n)=M_1\supset M_2 \supset \cdots \supset M_{m} \supset M_{m+1}=\emptyset,
\end{equation}
\begin{equation} \label{filtration of p}
\emptyset=\pi_1\subset \pi_2 \subset \cdots \subset \pi_{m} \subset \pi_{m+1}=\p,
\end{equation}
where $M_{i}\setminus M_{i+1}=M^+_i,$ $\pi_{i+1}\setminus \pi_{i}=M^-_i$ and the decomposition of $M(r,n)$ (resp. $\p$) is given by $M(r,n)= \bigcup_{j\in J} M^+_j$ (resp. $\p= \bigcup_{j\in J} M^-_j$).

Define the subsets $M^+_{\leq j}:= \bigcup_{i\leq j}M^+_i$ and $M^-_{\leq j}:= \bigcup_{i\leq j}M^-_i.$ Then the following holds.
\begin{lemma}
For each $j,$ there is an inclusion $M^-_{\leq j} \hookrightarrow M^+_{\leq j}.$
\end{lemma}
\proof
We prove the assertion by using the filtration \eqref{filtration of M}. The intersection $M_j^- \cap M^+_j$ is the fixed point $\{x_j\},$ so $M_j^- \cap (M_j \setminus M_{j+1})=\{x_j\}.$ This means $M_j^- \cap M_{j+1}= \emptyset$ since  $x_j\in M_j,$  $x_j \notin M_{j+1}$ and $M_{j+1}\subset M_j.$ Note that $M_{j+1}=M^+_{\geq j+1}$ so $M_j^- \subset M^+_{\leq j}$ for all $j.$ Hence we get the inclusion $M^-_{\leq j} \subset M^+_{\leq j}$ for all $j.$
\endproof

\begin{theorem} \label{iso homology groups}
The inclusion $M^-_{\leq j} \hookrightarrow M^+_{\leq j}$ induces isomorphisms of homology groups with integer coefficients $H_k(M^-_{\leq j}) \stackrel{\cong}\longrightarrow H_k(M^+_{\leq j})$ for all $j$ and all $k.$ In particular the inclusion $\p \hookrightarrow M(r,n)$ induces isomorphisms $H_k(\p) \stackrel{\cong}\longrightarrow H_k(M(r,n))$ for all $k.$
\end{theorem}

\proof
We prove the theorem by induction on $j.$

For $j=1$ the inclusion $M^-_1\hookrightarrow M^+_1$ induces an isomorphism of homology groups $H_k(M^-_1)\stackrel{\cong}{\longrightarrow} H_k(M^+_1)$ for all $k$ since $M^-_1$ and $M^+_1$ are isomorphic to affine spaces.

Suppose the inclusion $M^-_{\leq j}\hookrightarrow M^+_{\leq j}$ induces an isomorphism of homology groups $H_k(M^-_{\leq j})\stackrel{\cong}{\longrightarrow} H_k(M^+_{\leq j})$ for all $k$ and consider the homology long exact sequences of the pairs $(M^-_{\leq j+1},M^-_{\leq j})$ and $(M^+_{\leq j+1},M^+_{\leq j})$ respectively. We have the following diagram

\begin{center}
\begin{tikzpicture}
\matrix(m)[matrix of math nodes, row sep=2em, column sep=1.5em, text height=1.5ex, text depth=0.25ex]
{\cdots & H_{k+1}(M^-_{\leq j+1},M^-_{\leq j}) & H_k(M^-_{\leq j}) & H_k(M^-_{\leq j+1}) & \cdots \\
\cdots & H_{k+1}(M^+_{\leq j+1},M^+_{\leq j}) & H_k(M^+_{\leq j}) & H_k(M^+_{\leq j+1}) & \cdots \\};
\path[->,font=\scriptsize]
(m-1-1) edge (m-1-2)
(m-1-2) edge (m-1-3)
(m-1-2) edge (m-2-2)
(m-1-3) edge (m-1-4)
(m-1-3) edge node[auto] {$\cong$} (m-2-3)
(m-1-4) edge (m-1-5)
(m-1-4) edge (m-2-4)
(m-2-1) edge (m-2-2)
(m-2-2) edge (m-2-3)
(m-2-3) edge (m-2-4)
(m-2-4) edge (m-2-5);
\end{tikzpicture}
\end{center}
where the arrows are induced by inclusions. The middle arrow is an isomorphism by the induction hypothesis.

$M^-_{\leq j}$ is closed in $M^-_{\leq j+1}.$ This follows from \eqref{closure} by reversing the order of the inequality to meet that of the filtration \eqref{filtration of p}. From \cite[Proposition 2.22]{Hatcher} $$H_k(M^-_{\leq j+1},M^-_{\leq j}) \cong H_k(M^-_{\leq j+1}/M^-_{\leq j}, p),$$ where $p$ is the point at infinity. The quotient space $M^-_{\leq j+1}/M^-_{\leq j}$ is isomorphic to the one-point compactification of $M^-_{j+1}$ that is homeomorphic to the Thom space $T(M^-_{j+1})$ of $M^-_{j+1}$ \cite[Ex. 138]{Davis-Kirk}. It follows that $$H_k(M^-_{\leq j+1},M^-_{\leq j}) \cong H_k(T(M^-_{j+1}), p).$$
$M^-_{j+1}$ is isomorphic to an affine space so it is isomorphic to $N_{|{x_{j+1}}},$ the normal bundle to $x_{j+1}$ in $M^-_{j+1}.$ Hence
$$H_k(M^-_{\leq j+1},M^-_{\leq j}) \cong H_k(T(M^-_{j+1}), p)\cong H_k(T(N_{{|x_{j+1}}}), p).$$

Let us denote by $V(M^+_{j+1})$ the tubular neighborhood of $M^+_{j+1}$ in $M^+_{\leq j+1}.$ By excision we get the following isomorphism
$$H_k(M^+_{\leq j+1},M^+_{\leq j})= H_k(M^+_{\leq j}\cup V(M^+_{j+1}),M^+_{\leq j}) \cong H_k(V(M^+_{j+1}),\partial V(M^+_{j+1})),$$
where $\partial V(M^+_{j+1}) = V(M^+_{j+1})\setminus M^+_{j+1}.$ \\Denote by $N$ the normal bundle of $M^+_{j+1}.$ From the tubular neighborhood theorem $V(M^+_{j+1})$ is homeomorphic to $N$ and $M^+_{j+1}$ is homeomorphic to the zero section of $N,$ see \emph{e.g.} \cite{Bott-Tu}. Thus
$$H_k(V(M^+_{j+1}),\partial V(M^+_{j+1})) \cong H_k(N,N_0) \cong H_k(T(N),p),$$
where $N_0 \subset N$ is the complement of the zero section in $N$ and $T(N)$ is the Thom space of $N.$

Since the point $x_{j+1}$ is the deformation retract of $M^+_{j+1}$ then $N$ deformation retracts to $N_{x_{j+1}},$ the fibre of $N$ at $x_{j+1}.$ Moreover, $N_{x_{j+1}}$ is a deformation retract of $N_{{|x_{j+1}}}.$ It follows that $H_k(T(N_{{|x_{j+1}}}), p)\cong H_k(T(N_{x_{j+1}}), p)$ and hence $H_k(M^-_{\leq j+1},M^-_{\leq j})\cong H_k(M^+_{\leq j+1},M^+_{\leq j})$ for all $j$ and all $k.$ By the five lemma we conclude the proof.
\endproof

\begin{lemma}\label{M 1-connected}
$M(r,n)$ is simply connected.
\end{lemma}

\proof
There exists a unique cell  $M^+_1$ of maximal dimension that is open in $M(r,n)$ (see Remark \ref{rem}). Since $M^+_1 \subset M(r,n)$ is isomorphic to an affine space then $\pi_1(M^+_1)=0$.
By \cite[Theorem 12.1.5]{Cox} the inclusion $M^+_1 \hookrightarrow M(r,n)$ induces a surjective map $\pi_1(M^+_1)\rightarrow \pi_1(M(r,n)).$ Hence $\pi_1(M(r,n))=0.$
\endproof

\begin{remark} Note that $M(r,n)$ is irreducible since it is non-singular and connected.
\end{remark}

\begin{lemma}\label{p 1-connected}
$\p$ is simply connected.
\end{lemma}

\proof
Let $M^-_m \hookrightarrow \p$ be the inclusion of the cell of maximal dimension into $\p.$ By \cite[Theorem 3(a)]{Car.Som} the induced map  $\pi_1(M^-_m) \rightarrow \pi_1(\p)$ is an isomorphism. $\pi_1(M^-_m)=0$ since $M^-_m$ is isomorphic to an affine space. Hence $\p$ is simply connected.
\endproof

\begin{theorem}\label{homotopy equivalence}
$\p$ is homotopy equivalent to $M(r,n).$
\end{theorem}

\begin{proof}
The inclusion $\p \hookrightarrow M(r,n)$ induces morphisms of homotopy groups and we have the following diagram
\begin{center}
\hspace{-3.5mm}\begin{tikzpicture}
\matrix(m)[matrix of math nodes,
row sep=2em, column sep=1.5em,
text height=1.5ex, text depth=0.25ex]
{\cdots & \pi_{k+1}(M(r,n),\p) & \pi_k(\p) & \pi_k(M(r,n)) & \cdots\\
\cdots& H_{k+1}(M(r,n),\p) & H_k(\p) & H_k(M(r,n)) & \cdots\\};
\path[->,font=\scriptsize]
(m-1-1) edge (m-1-2)
(m-1-2) edge (m-1-3)
(m-1-2) edge (m-2-2)
(m-1-3) edge (m-1-4)
(m-1-3) edge (m-2-3)
(m-1-4) edge (m-1-5)
(m-1-4) edge (m-2-4)
(m-2-1) edge (m-2-2)
(m-2-2) edge (m-2-3)
(m-2-3) edge node[auto] {$\cong$} (m-2-4)
(m-2-4) edge (m-2-5);
\end{tikzpicture}
\end{center}
where the isomorphism is from Theorem \ref{iso homology groups} and $H_{k}(M(r,n),\p)=0$ for all $k.$

From Lemma \ref{M 1-connected} and Lemma \ref{p 1-connected}, both $M(r,n)$ and $\p$ are $1$-connected, then using Hurewicz's theorem we get
$$\pi_2(M(r,n))\cong H_2(M(r,n)),$$
$$\pi_2(\p)\cong H_2(\p).$$
This yields the following diagram
\begin{center}
\begin{tikzpicture}
\matrix(m)[matrix of math nodes,
row sep=2em, column sep=1.3em,
text height=1.5ex, text depth=0.25ex]
{\pi_2(\p) & \pi_2(M(r,n)) & \pi_2(M(r,n),\p) &0\\
H_2(\p) & H_2(M(r,n)) & 0 \\};
\path[->,font=\scriptsize]
(m-1-1) edge (m-1-2)
(m-1-2) edge (m-1-3)
(m-1-2) edge node[auto] {$\cong$} (m-2-2)
(m-1-3) edge (m-1-4)
(m-1-3) edge (m-2-3)
(m-1-1) edge node[auto] {$\cong$} (m-2-1)
(m-2-1) edge node[auto] {$\cong$} (m-2-2)
(m-2-2) edge (m-2-3);
\end{tikzpicture}
\end{center}
Hence we get $\pi_2(M(r,n))\cong \pi_2(\p)$ and $\pi_2(M(r,n),\p)=0;$ then the pair $(M(r,n),\p)$ is 2-connected. By the relative Hurewicz theorem we have $$\pi_3(M(r,n),\p)\cong H_3(M(r,n),\p)=0.$$ Iterating this process it follows that for every $k$,
$$\pi_k(M(r,n),\p)\cong H_k(M(r,n),\p)=0.$$
Hence the long exact sequence of homotopy groups reduces to
\begin{center}
\begin{tikzpicture}
\matrix(m)[matrix of math nodes,
row sep=2em, column sep=1.3em,
text height=1.5ex, text depth=0.25ex]
{\pi_k(\p) & \pi_k(M(r,n))\\};
\path[->,font=\scriptsize]
(m-1-1) edge node[auto] {$\cong$} (m-1-2);
\end{tikzpicture}
\end{center}
for all $k.$ Finally Whitehead's theorem (see \emph{e.g.} \cite[page 370]{Rotman}) concludes the proof.
\end{proof}

\section{Moduli on toric surfaces} \label{gener}

In this section, we provide a generalization to the study of the moduli space of framed sheaves on a nonsingular projective toric surface $S.$ We consider the framing sheaf to be supported on a big and nef divisor $D \subset S$. Furthermore we assume that there exists a projective morphism of toric surfaces $p:S \to \mathbb P^2$ of degree 1.

First we study the fixed point locus of the moduli space on any toric surface $X$. Then we restrict ourselves to the moduli space $M(S)$ on the toric surface $S$ and construct a projective morphism from $M(S)$ to $M_0(r,n),$ the moduli space of ideal instantons introduced in section \ref{generalities}. Using this projective morphism we define a compact subvariety $\widetilde{N}$ of $ M(S)$ and study some topological properties, namely singular homology and homotopy equivalence between $M(S)$ and $\widetilde{N}.$

\subsection{Torus action and fixed points}
In this section we will construct an action of a torus $T$ on the moduli space of framed torsion-free sheaves $M(X)$ on X. Here $X$ is a nonsingular projective toric surface. This moduli space is a quasi-projective variety as shown in \cite{BruzMark}. We will show that the action of $T$ gives rise to a finite set of fixed points $(M(X))^{\C^*}$ and $(M^{\mu ss}(X))^T.$

We denote by $M^{\mu ss}(X)$ the space of semistable framed sheaves on $X$ defined in \cite[Definition 4.5]{BruzMarkTikh}, and by $M^{\mu\text{-poly}}(X)$ the moduli space of polystable framed sheaves on $X.$

As shown in \cite{BruzMarkTikh}, there is a projective morphism $\gamma$ from $M(X)$ onto $M^{\mu ss}(X)$ given as follows
\begin{align*}
\gamma: \quad\quad M \:\: &\longrightarrow M^{\mu ss}(X)= \coprod_{k\geq 0} M^{\mu \text{-poly}}(r,\xi,c_2-k,\delta)\times \Sym^k(X\setminus D)\\
(\curly E, \alpha) &\longmapsto \bigg(\big(gr^\mu (\curly E,\alpha)\big)\dual \dual, \Supp\frac{(gr^\mu\curly E)\dual \dual}{gr^\mu \curly E}\bigg),
\end{align*}
where the support of a sheaf $\curly F$ counted with multiplicities is denoted by $\Supp\curly F.$

Note that the double dual of a $\mu$-semistable framed torsion free sheaf is a $\mu$-polystable framed locally free sheaf, and the support of $\big((gr^\mu\curly E)\dual\dual / gr^\mu \curly E\big)$ is a point in $\Sym^k(X\setminus D)$ where $k=c_2(gr^{\mu}\curly E)-c_2\big((gr^{\mu}\curly E)\dual\dual\big).$

Now assume that the action of the $2$-dimensional algebraic torus $T^2$ on $X$ gives rise to a finite set of isolated fixed points $X^{T^2}=\{x_1,\dots,x_n\},$ that the framing divisor $D$ is toric, i.e., stable under the action of $T^2,$ and let the framing sheaf $\curly F$ be locally free on $D.$ Suppose we have an action of an $r$-dimensional torus $T^r$ on $\curly F.$ Then this induces an action of an $(r+2)$-dimensional torus $T$ on $M(X)$ and on $M^{\mu ss}(X).$

Let us consider the action of $T^2 \cong \C^*\times \C^*$ on $X.$ Then for any element $(t_1,t_2)$ of $T^2$ one has an automorphism $h_{t_1,t_2}$ of $X.$
$$h_{t_1,t_2}:  X \longrightarrow X$$

The action of $T^2$ on the sheaf $\curly E$ is defined by by taking the inverse image via the automorphism $h$: $\curly E \mapsto \curly E'=(h_{t_1,t_2}^{-1})^*\curly E.$

To define the action on the framing $\alpha$ let us consider the torus $T^r \cong \C^*\times \cdots \times \C^*$ ($r$-times) that acts on the framing sheaf as follows. For an element $(f_1,\dots,f_r) \in T^r$ let $F_{f_1,\dots,f_r}$ be the isomorphism of $\curly F$

$$F_{f_1,\dots,f_r}: \curly F \longrightarrow \curly F$$

Then the action of $T \cong \C^*\times \cdots \times \C^*$ ($(r+2)$-times) on a pair $(\curly E,\alpha)\in M(X)$ is defined by
$$(t_1,t_2,f_1,\dots,f_r)\cdot (\curly E,\alpha):= \big((h_{t_1,t_2}^{-1})^* \curly E, \alpha'\big),$$ where $\alpha'$ is given by the composition of the following maps:
$$\curly E'|_{D}=(h_{t_1,t_2}^{-1})^* \curly E|_{D}\xrightarrow{(h_{t_1,t_2}^{-1})^*\alpha}(h_{t_1,t_2}^{-1})^* \curly F\longrightarrow \curly F\xrightarrow{F_{f_1,\dots, f_r}} \curly F,$$
where the middle arrow is given by the action of $T^2$ under which $\curly F$ is stable.

The $T$ action is defined in a similar way on $M^{\mu ss}(X),$ namely the torus action on $M^{\mu\text{-poly}}(X)$ is the same as the action on $M(X)$ and the action of $T^2$ on $\Sym^k (X\setminus D)$ is induced from that on $X$ since $D$ is $T^2$-invariant. It is possible to check that the map $\gamma$ is equivariant, i.e., $\gamma$ commutes with the torus action. Indeed for a framed shead $(\curly E,\alpha) \in M(X),$ the following two compositions agree:
\begin{multline*}
(\curly E,\alpha) \stackrel{\gamma}\longmapsto \bigg(\big(gr^\mu (\curly E,\alpha)\big)\dual \dual, \Supp\frac{(gr^\mu\curly E)\dual \dual}{gr^\mu \curly E}\bigg) \xmapsto{\text{torus action}} \\
\bigg(\big((h_{t_1,t_2}^{-1})^* (gr^\mu \curly E) \dual \dual,(gr^\mu \alpha)' \big), h_{t_1,t_2}\Supp\frac{(gr^\mu\curly E)\dual \dual}{gr^\mu \curly E}\bigg),
\end{multline*}

\begin{multline*}
(\curly E,\alpha) \xmapsto{\text{torus action}} \big((h_{t_1,t_2}^{-1})^* \curly E, \alpha'\big)\stackrel{\gamma}\longmapsto\\
\bigg(\big(gr^\mu \big((h_{t_1,t_2}^{-1})^* \curly E, \alpha'\big)\big)\dual \dual, \Supp\frac{\big(gr^\mu\big((h_{t_1,t_2}^{-1})^* \curly E\big)\big)\dual \dual}{gr^\mu\big((h_{t_1,t_2}^{-1})^* \curly E\big)}\bigg).
\end{multline*}

This is because the torus action is compatible with the Jordan-H\"{o}lder filtration and with the double dual. The torus action on the support of a sheaf gives the support of a sheaf acted on by the torus action.

\begin{lemma}\label{fixpoints}
The fixed points set $(M(X))^T$ consists of finitely many points. The same holds for $(M^{\mu ss}(X))^T.$
\end{lemma}

\proof
A framed sheaf $(\curly E,\alpha)$ is fixed if it can be written in the form $(\curly E,\alpha)=(\curly E_1,\alpha_1)\oplus\dots\oplus(\curly E_r,,\alpha_r)$ such that $\curly E_i=\curly I_i\otimes \curly O(C_i).$ Here $C_i$ is a $T^2$-invariant divisor that does not intersect $D,$ $\curly I_i$ is the ideal sheaf of a zero dimensional subscheme $Z_i$ in $X\setminus D,$ and $\alpha_i$ is an isomorphism $\curly E_i|_D \xrightarrow{\:\cong\:} \curly F_i,$ where $\curly F_i$ are rank one locally free subsheaves of $\curly F$ supported on $D$ such that the direct sum $\bigoplus_{i=1}^r \curly F_i=\curly F.$ Note that the sheaf $\curly F$ decomposes into such a direct sum since it is locally free on a toric divisor in $X$ that is a a smooth curve.\\
    The ideal sheaves $\curly I_i$ are fixed if they are generated by monomials in the homogeneous coordinate ring (Cox ring) of $X.$ These monomials are finite and hence the ideal sheaves $\curly I_i$ form a finite family. Moreover, the Picard group of a compact projective variety is generated by a finite number of divisors \cite[Corollary 2.5]{Oda}. Hence  $\curly E_i=\curly I_i\otimes \curly O(C_i)$ form a finite family. As a result the fixed points set of $M(X)$ is finite.

Now regarding $M^{\mu \text{-poly}}(r,\xi,c_2-k,\delta)$ as an open subset of $M(X)$ with the corresponding invariants, the set of its fixed points is finite. A framed sheaf $(\curly E,\alpha)\in M^{\mu \text{-poly}}(r,\xi,c_2-k,\delta)$ is fixed if is it given as above, hence $\curly E= \bigoplus_{i=1}^r \curly E_i = \bigoplus_{i=1}^r(\curly I_i\otimes \curly O_X(C_i)).$ Since $\curly E$ is locally free then so is $\curly I_i$ for each $i.$ We have $\curly I_i\cong \curly I_i \dual \dual \cong \curly O_X,$ hence $\curly E= \bigoplus_{i=1}^r\curly O_X(C_i).$

The fixed points set of $\Sym^k(X\setminus D)$ is finite since $X^T$ is. It follows that $(M^{\mu ss}(X))^T$ is finite.
\endproof

\begin{remark}
From Lemma \ref{fixpoints} it follows that $N:=\gamma^{-1} \big((M^{\mu ss}(X))^T\big)$ contains all the fixed points of $M(X)$ since $\gamma$ is equivariant. This fact will not be needed in what follows.
\end{remark}

\subsection{Constructing a projective morphism}

In this section we will restrict ourselves to a toric surface $S$ admiting a projective morphism of toric surfaces $p:S \to \mathbb P^2$ of degree 1. An example of these surfaces is the iterated toric blowup of $\mathbb P^2,$ i.e, the iterated blowup along a set of points that are fixed under the torus action. We will also assume that the framing sheaf is free $\curly F= \Or_D.$ Note that under the assumptions above the direct image of the framing sheaf on $D$ is isomorphic to the framing sheaf on $l_{\infty}.$

We will construct a projective morphism from the moduli space $M(S)$ of framed torsion-free sheaves on $S$ to $M_0(r,n),$ the moduli space defined in section \ref{generalities}. To this end we will follow the construction in \cite[Appendix F]{NY}.

Next we will show that there exists a morphism between $M(S)$ and $M_0(r,n)$ and prove it is projective.

Let $\mathcal M(r,c_1,n)$ be the moduli space of $H$-stable sheaves $\curly E$ on $\mathbb P^2$ with rank $r:=\rk \curly E,$ first Chern class $c_1:=c_1(\curly E)$ and discriminant $n:= c_2(\curly E)-\frac{r-1}{2r} c_1(\curly E)^2.$ Assuming that the degree and the rank of a stable sheaf are coprime 
the moduli space consists of $\mu$-stable sheaves. We define $\mathcal M_{loc}(r,c_1,n)$ the subscheme of $\mathcal M(r,c_1,n)$ consisting of $\mu$-stable locally free sheaves.

Let $\widetilde{\mathcal M}(r,p^*c_1+kC,n)$ be the moduli space of $(H-\epsilon C)$-stable sheaves $\curly E$ on $S$ of rank $r,$ first Chern class $c_1(\curly E)=p^*c_1+kC,$ and discriminant $n:= c_2(\curly E)-\frac{r-1}{2r} c_1(\curly E)^2.$ Here $C$ is a (reducible) divisor on $S$ that does not intersect $D,$ $k$ is an integer and $\epsilon $ is sufficiently small.

\subsubsection{Uhlenbeck compactification of $\mathcal M_{loc}(r,c_1,n)$}

We define the Uhlenbeck compactification of the the moduli space of locally-free sheaves $\mathcal M_0(r,c_1,n)$ as follows.
$$ \mathcal M_0(r,c_1,n):=\bigsqcup_l \mathcal M_{loc}(r,c_1,n-l)\times \Sym^l(\mathbb P^2).$$

Li proved the following theorem in the general setting of moduli spaces on projective surfaces, see \cite{li, li-jun}. Here is Li's theorem for our moduli spaces.
\begin{theorem}
\begin{enumerate}
\item $\mathcal M_0(r,c_1,n)$ is a projective scheme.
\item There is a projective morphism
\begin{align}\label{projective morphism Li}
\pi\colon \hspace{0.3cm} \mathcal M(r,c_1,n) &\longrightarrow \mathcal M_0(r,c_1,n) \\ \notag
\curly E \hspace{.15cm} &\longmapsto
   (\curly E{\dual\dual}, \Supp(\curly E{\dual\dual}/\curly E)).
\end{align}
\end{enumerate}
\end{theorem}

When the first Chern class is zero, the morphism \eqref{projective morphism Li} reduces to the morphism \eqref{projective morphism} defined in section \ref{generalities}.
This morphism will allow us define a new morphism $\widetilde{\pi}$ from the moduli space of $\mu$-stable sheaves on $S$ to $\mathcal M_0(r,c_1,n).$

\subsubsection{Defining a morphism $\widetilde{\pi}$}
In this section we will assume that the first Chern class $c_1=k\cdot C$ and $0\leq k <r.$ For a sufficiently large $l$ such that  $l=k$ modulo $r,$ we define the following morphism
\begin{align*}
\beta: \quad \widetilde{\mathcal M}(r,k,\tilde{n}) &\longrightarrow \mathcal M(r,n)\\
\curly E \quad &\longmapsto p_*\curly E(-lC).
\end{align*}

The composition of $\beta$ with the morphism $\pi$ defined in \eqref{projective morphism Li} gives the following morphism
\begin{align}
\widetilde{\pi}\colon \hspace{0.3cm} \widetilde{\mathcal M}(r,k,\tilde{n}) &\to \mathcal M_0(r,n) \label{p for non framed} \\ \notag
\curly E \hspace{.15cm} &\mapsto
   \bigg((p_* \curly E(-lC)){\dual\dual}, \Supp \frac{(p_*\curly E(-lC)){\dual\dual}}{p_*\curly E(-lC)}\bigg).
\end{align}

In the next section we will show how the morphism \eqref{p for non framed} restricts to a morphism between moduli spaces of framed sheaves and show it is projective.

\begin{remark}
Note that the definition of the morphism $\widetilde{\pi}$ relies on Lemma \ref{lemma1}. Using this lemma together with the Grothendieck-Riemann-Roch theorem, one can compute $n$ in terms of $k$ and $\tilde{n}$. Moreover, by the same arguments  one can show that the morphism $\widetilde{\pi}$ does not depend on the choice of $l.$ This holds since we are assuming that $l$ is equal to $k$ modulo $r.$
\end{remark}

\subsubsection{$\widetilde{\pi}$ for framed sheaves}
Before getting to the definition of the morphism $\widetilde{\pi}$ for framed sheaves, we will need a few results.

\begin{lemma}[Lemma F.19 \cite{NY}]\label{lem}
Denote by $\delta_1$ the leading coefficient of the polynomial $\delta.$ Assume that $\delta_1 \ll 1,$ then the following holds
\begin{enumerate}
\item For a semistable framed sheaf $(\curly E, \alpha),$ $\curly E$ is torsion-free.
\item 
All torsion-free $\mu$-semistable sheaves are $\mu$ stable.
\end{enumerate}
\end{lemma}

This lemma yields Lemma F.20 in \cite{NY}. This works in our case too. Hence using Lemma \ref{lem} we have the following result.

\begin{lemma} Assume that $\delta_1 \ll 1.$ If $\epsilon >0$ depending on $\delta_1$ and $\ch(E)$ is sufficiently small, then $(\curly E, \alpha)$ is semistable with respect to $H-\epsilon C$ and $\delta$ if and only if $(p_*\curly E(-lC),\alpha)$ is semistable with respect to $H$ and $\delta.$
In particular, the moduli space of framed torsion-free sheaves on $S$ is contained in the moduli space of semistable pairs on $S.$ \end{lemma}

This lemma states that $\beta$ sends semistable framed sheaves on $S$ to semistable framed sheaves on $\mathbb P^2.$ Hence it extends to a morphism between moduli spaces of semistable sheaves on $S$ and on $\mathbb P^2.$ Since these moduli spaces are projective, $\beta$ is a projective morphism. In other words, we have a projective morphism, that we denote $\beta$ for simplicity, from the moduli space of framed sheaves on $S$ to the moduli space of framed sheaves on $\mathbb P^2.$
\begin{align}
\beta: \quad M(S) &\longrightarrow M(r,n) \label{beta framed}\\ \notag
(\curly E, \alpha) &\longmapsto \big(p_*\curly E(-lC), \phi\big),
\end{align}
where the framing $\phi$ is given by
\begin{equation}\label{phi}
\phi: p_*\curly E(-lC)|{\linf} \xrightarrow{\: \cong \:} p_*\curly \Oo_{S}^{\oplus r}(-lC)|{\linf} \cong \Or_{\linf},
\end{equation}since $p_*\curly \Oo_{S}(-lC)$ is isomorphic to $\Oo_{\mathbb P^2}.$

On the other hand, the morphism $\pi$ defined in \eqref{projective morphism} is projective, so the composition $\pi \circ \beta$ is projective. Hence we have proved the following theorem.

\begin{theorem} \label{proj mor}
There is a projective morphism
\begin{align}
\widetilde{\pi}\colon \hspace{0.3cm} M(S) \: &\to M_0(r,n) \label{p for framed} \\ \notag
(\curly E,\alpha) \hspace{.15cm} &\mapsto
   \bigg(\big((p_* \curly E(-lC)){\dual\dual}, \phi\big) , \Supp \frac{(p_*\curly E(-lC)){\dual\dual}}{p_*\curly E(-lC)}\bigg),
\end{align}
where the framing $\phi$ is given by \eqref{phi}.
\end{theorem}

\begin{remark}
The map $\widetilde{\pi}$ is equivariant under the torus action since $\beta$ is.
\end{remark}

In the next section, using this morphism we will study some topological properties of the moduli space of framed torsion-free sheaves on $S.$

\subsection{Some topological properties}\label{decomp M and gamma}

In this section, we define the subvariety $\widetilde N$ of $M(S)$ that is the inverse image by $\widetilde \pi$ of the fixed point of $M_0(r,n).$ We show that both $M(S)$ and $\widetilde{N}$ admit a filtrable Bia{\l}ynicki-Birula decomposition. It is enough to consider the action of a regular one-parameter subgroup $\lambda \subset T$ as in Section \ref{decomp M and pi}. Therefore, we will consider the action of $\C^*$ on $M(S)$ and on $\widetilde{N}.$

Having established the projective morphism $\widetilde{\pi},$ we define the inverse image $$\widetilde{N}:=\widetilde{\pi}^{-1}(n[0])$$ of the fixed point $n[0]$ of $M_0(r,n).$ Then $\widetilde{N}$ is a compact subvariety of $M(S).$

\begin{proposition} \label{the limit exists}
For every element $x \in M(S),$ the limit $\lim_{t\rightarrow 0} t\cdot x$ exists and lies in $M(S)^{\C^*}.$
\end{proposition}

\begin{proof}
The proof of this proposition goes through the same lines as that of Proposition \ref{limit exists}.
\end{proof}

\begin{theorem} \label{BB decomp of M(S)}
$M(S)$ admits a Bia{\l}ynicki-Birula plus-decomposition into affine spaces. Moreover, this decomposition is filtrable.
\end{theorem}
\proof
The theorem follows by using Proposition \ref{the limit exists} and proceeding as in the proof of Theorem \ref{BB decomp of M}.
\endproof

\begin{corollary}
$M(S)$ is simply connected and irreducible.
\end{corollary}

\proof
The simple connectedness of $M(S)$ is shown using the same arguments as in the proof of Theorem \ref{M 1-connected}.

$M(S)$ is irreducible since it is nonsingular and connected.
\endproof

Now suppose $M(S)^{\C^*}=\{x_i \mid i=1,\dots, m \}$ and denote the cells of the decomposition by $$M(S)^+_i:=M(S)^+_i=\{x\in M(S) \mid \lim_{t \to 0} t\cdot x=x_i \}.$$

The limit $\lim_{t \to \infty} t\cdot x$ does not exist for all $x \in M(S).$ We will only consider the points of $M(S)$ such that the limit exists. These points define the subspaces $$M(S)^-_i:=M(S)^-_i=\{x\in M(S) \mid \lim_{t \to \infty} t\cdot x \text{ exists and equals } x_i \}.$$

\begin{theorem} \label{BB decomp of N}
The subvariety $\widetilde{N}$ is the union $\widetilde{N}=\bigcup_i M(S)_i^-.$ Moreover, this is a filtrable decomposition of $\widetilde{N}$ into affine spaces.
\end{theorem}

\proof
From \cite[Theorem 3.5(3)]{NY} any $x\neq n[0]$ in $M_0(r,n)$ the limit $\lim_{t\to \infty}t \cdot x$ does not exist. Hence by the projective morphism $\tilde{\pi}$ we deduce that $\widetilde{N}=\bigcup_i M(S)_i^-.$
To show this is a filtrable decomposition, the proof goes through the same lines as that of Theorem \ref{union of minus cells}.
\endproof

Let $J=\{1,2,\dots, m\}$ be the set of indices so that the fixed points set $X^{\C^*}=\{x_j \mid j\in J\}$ be ordered to yield the following filtrations
$$M(S)=M(S)_1\supset M(S)_2 \supset \cdots \supset M(S)_{m} \supset M(S)_{m+1}=\emptyset,
$$
$$\emptyset=\gamma_1\subset \gamma_2 \subset \cdots \subset \gamma_{m} \subset \gamma_{m+1}=\widetilde{N},$$
where $M(S)_{i}\setminus M(S)_{i+1}=M(S)^+_i,$ $\gamma_{i+1}\setminus \gamma_{i}=M(S)^-_i$ and the decomposition of $M(S)$ \big(resp. $\widetilde{N}$\big) is given by $M(S)= \bigcup_{j\in J} M(S)^+_j$ \big(resp. $\widetilde{N}= \bigcup_{j\in J} M(S)^-_j$\big).

Define the subsets $M(S)^+_{\leq j}:= \bigcup_{i\leq j}M(S)^+_i$ and $M(S)^-_{\leq j}:= \bigcup_{i\leq j}M(S)^-_i.$ Then all the results in Section \ref{more topology} hold and the proofs of the following are basically the same.

\begin{lemma} \label{lemma inclusions for M(S)}
For each $j,$ there is an inclusion $M(S)^-_{\leq j} \hookrightarrow M(S)^+_{\leq j}.$
\end{lemma}

\begin{theorem}\label{iso homology groups M(S)}
The inclusion $M(S)^-_{\leq j} \hookrightarrow M(S)^+_{\leq j}$ induces isomorphisms of homology groups with integer coefficients $H_k(M(S)^-_{\leq j}) \stackrel{\cong}\longrightarrow H_k(M(S)^+_{\leq j})$ for all $j$ and all $k.$ In particular the inclusion $\widetilde{N} \hookrightarrow M(S)$ induces isomorphisms $H_k(\widetilde{N}) \stackrel{\cong}\longrightarrow H_k(M(S))$ for all $k.$
\end{theorem}

\proof
Using Lemma \ref{lemma inclusions for M(S)} the proof is the same as that of Theorem \ref{iso homology groups}.
\endproof

\begin{lemma}
$\widetilde{N}$ is simply connected.
\end{lemma}

\begin{theorem}\label{homotopy equivalence M(S)}
$\widetilde{N}$ is homotopy equivalent to $M(S).$
\end{theorem}

\section{Concluding remarks}

The results of Section \ref{more topology} and \ref{gener} hold for all the moduli spaces on nonsingular projective toric surfaces $X$ having a projective morphism onto $M_0(r,n)$ which is equivariant under the torus action. Indeed, to generalize these results it is enough to construct a projective morphism
$$\widetilde{\pi} \colon  M(X) \to M_0(r,n).$$
However, this may not be an easy task.

As we have seen in the last section, to construct such a morphism we have considered the morphism
\begin{align} \label{beta}
\beta: \quad M(X) &\longrightarrow M(r,n) \notag \\
\curly E \quad &\longmapsto p_*\curly E(-lC),
\end{align}
for a sufficiently large $l.$

It is to notice that when the morphism $p$ is not of degree 1, the direct image of the structure sheaf $p_*\curly \Oo_X(-lC)$ is not in general isomorphic to $\Oo_{\mathbb P^2}.$ In this case a morphism defined as in \eqref{beta} does not do the job since we end up in a moduli space different from $M(r,n).$

The reason for constructing a morphism $\widetilde{\pi}$ lies in the fact that the moduli space $M_0(r,n)$ has a unique fixed point $n[0]$ and the fiber over it defines a compact subvariety containing all the fixed points of the moduli space we are considering. Moreover $M_0(r,n)$ has a description into ADHM data that was intensively used in the proofs of the results in the last two sections.

\appendix
\section{Useful statements}

\begin{lemma}\label{lemma}
Let $X$ be a smooth projective surface and let $$0\ra \curly F\ra \curly F' \ra \curly Q\ra 0,$$ be an exact sequence of sheaves of $\Oo_X$-modules where $\curly F$ is torsion free, $\curly F'$ is reflexive (equivalently locally free since every reflexive sheaf on a surface is locally free) of the same rank and $\curly Q$ is zero dimensional; then $\curly F'\iso \curly F{\dual\dual}.$
\end{lemma}
\begin{proof}
Dualizing the exact sequence $0\ra \curly F\ra \curly F' \ra \curly Q\ra 0,$ \textit{i.e.} applying the functor $\curly Hom(\cdot ,\Oo_{X}),$ we get
\[
0 \ra \curly Hom(\curly F',\Oo_{X})  \ra \curly Hom(\curly F,\Oo_{X}) \ra \curly Ext^1(\curly Q, \Oo_{X})\ra 0.
\]
Note that $\curly Hom( \curly Q,\Oo_X)=0,$ and $\curly Ext^1( \curly F, \curly G)=0$ for any coherent sheaf $\curly G$, where $\curly F$ is locally free~\cite[Chap. III, Ex. 6.5(a)]{Hartshorne}, in particular $\curly Ext^1(\curly F',\Oo_{X})$ vanishes.
We then get
\[
0 \ra \curly F'{\dual} \ra \curly F{\dual} \ra \curly Q' \ra 0,
\]
where $\curly Q':=\curly Ext^1(\curly Q, \Oo_{X})$ is a zero dimensional sheaf. Call $S$ the support of $\curly Q'$ then $S \subset X$ is a closed subset of codimension $2.$ \\Since $\curly F{\dual}$ and $\curly F'{\dual}$ are reflexive then by \cite[Prop.1.6]{Hartshorne1} they are normal. Hence for every open $U\subset X,$ we have
\begin{center}
\begin{tikzpicture}
\matrix(m)[matrix of math nodes, row sep=2em, column sep=1.5em, text height=1.5ex, text depth=0.25ex]
{0& \curly F'{\dual}(U)& \curly F{\dual}(U)& \curly Q'(U)&0\\
0&\curly F'{\dual}(U\setminus S)& \curly F{\dual}(U\setminus S) & \curly Q'(U\setminus S)& 0\\};
\path[->,font=\scriptsize]
(m-1-1) edge (m-1-2)
(m-1-2) edge (m-1-3)
(m-1-2) edge node[auto] {$\cong$} (m-2-2)
(m-1-3) edge (m-1-4)
(m-1-3) edge node[auto] {$\cong$} (m-2-3)
(m-1-4) edge (m-1-5)
(m-2-1) edge (m-2-2)
(m-2-2) edge (m-2-3)
(m-2-3) edge (m-2-4)
(m-2-4) edge (m-2-5);
\end{tikzpicture}
\end{center}
Note that $\curly Q'(U\setminus S)=0$ since $\curly Q'$ is supported on $S.$ Thus the injective morphism $\curly F'{\dual}(U\setminus S) \rightarrow \curly F{\dual}(U\setminus S)$ is an isomorphism. This yields an isomorphism
$\curly F'{\dual}(U) \xrightarrow{\cong} \curly F{\dual}(U)$ and we get $\curly Q'(U)=\curly Ext^1(\curly Q, \Oo_{X})(U)=0$  for every open subset $U \subset X.$ Hence we find $\curly Q'=\curly Ext^1(\curly Q, \Oo_{X})=0$ and $\curly F'{\dual}  \xrightarrow{\cong} \curly F{\dual}.$ Dualizing again one gets $\curly {F'{\dual\dual}}  \iso \curly F{\dual\dual} .$ Since $\curly F'$ is reflexive it follows that $\curly {F'}  \iso \curly F{\dual\dual}.$
\end{proof}

\begin{lemma} \label{lemma1}
There is an integer $l_0$ such that for any $\curly E \in \widetilde M(r, k, n)$ and any $l\geq l_0,$ we have:
\begin{enumerate}
  \item $R^1p_*\curly E(-lC)=0$
  \item There is a canonical inclusion $\curly E(-lC) \hookrightarrow p^*\big(p_*\curly E(-lC)\big)\dual\dual.$
  \item $\big(p_*\curly E(-lC)\big)\dual\dual \cong (p_*\curly E)\dual\dual.$
\end{enumerate}
\end{lemma}

\begin{proof}
Note that $\curly E(-lC)=\curly E(l)$ since the ideal sheaf of the exceptional divisor $C$ on $S$ is given by $\Oo_{S}(-C)=~\Oo_{S}(1)$ see ~\cite[IV, Lemma 4.1.(b)]{Fulton-Lang}.
\begin{enumerate}
  \item The proof can be found in ~\cite[III. Theorem 8.8.(c)]{Hartshorne}
  \item From ~\cite[III. Theorem 8.8.(a)]{Hartshorne} the natural map $p^*p_*\curly E(-lC)\longrightarrow \curly E(-lC)$ is surjective, hence one gets an exact sequence:
      $$0\longrightarrow \curly T \longrightarrow p^*p_*\curly E(-lC)\longrightarrow \curly E(-lC)\longrightarrow 0.$$
      Since $p^*p_*\curly E(-lC)$ and $\curly E(-lC)$ have the same rank, it follows that $\curly T$ is a torsion sheaf, but $p^*p_*\curly E(-lC)$ is torsion free, then $\curly T=0$ and $p^*p_*\curly E(-lC)\longrightarrow \curly E(-lC)$ is an isomorphism.

    On the other hand, there is a canonical inclusion $p^*(p_*\curly E(-lC))\hookrightarrow p^*\big(p_*\curly E(-lC)\big)\dual\dual,$  so we have the following diagram
    \begin{center}
    \begin{tikzpicture}[>=angle 90]
    \matrix(a)[matrix of math nodes,
    row sep=3em, column sep=2.5em,
    text height=1.5ex, text depth=0.25ex]
    {p^*p_*\curly E(-lC) & p^*\big(p_*\curly E(-lC)\big)\dual\dual\\
    \curly E(-lC)& \\};
    \path[right hook->](a-1-1) edge (a-1-2);
    \path[->](a-1-1) edge node[above, sloped] {$\simeq$} (a-2-1);
    \path[dotted,->](a-2-1) edge (a-1-2);
    \end{tikzpicture}
    \end{center}

Hence the dotted arrow is injective and we get the inclusion $\curly E(-lC) \hookrightarrow p^*\big(p_*\curly E(-lC)\big)\dual\dual.$

  \item Using the inclusion above, and taking the cokernel, one gets the exact sequence
    $$0 \longrightarrow \curly E(-lC) \longrightarrow p^*\big(p_*\curly E(-lC)\big)\dual\dual \longrightarrow \curly Q \longrightarrow 0.$$
    Tensoring by $\Oo_{S}(lC)$ and applying the direct image, the above exact sequence gives
    $$0 \longrightarrow p_*\curly E \longrightarrow p_*\big(p^*\big(p_*\curly E(-lC)\big)\dual\dual \otimes \Oo_{S}(lC)\big) \longrightarrow \cdots  $$

    On the other hand, $p_*\Oo_{S}(lC)=p_*\Oo_{S}(-l)= \Oo_{S}$ see ~\cite[page 76]{Buchdahl}. Hence, substituting and taking the quotient in the above exact sequence, one gets the short exact sequence
    $$0 \longrightarrow p_*\curly E \longrightarrow \big(p_*\curly E(-lC)\big)\dual\dual  \longrightarrow \curly Q' \longrightarrow 0.$$

    Notice that $p_*\curly E$ is torsion free and $ \big(p_*\curly E(-lC)\big)\dual\dual$ is locally free of the same rank, $\curly Q'$ is supported on points, then using Lemma ~\ref{lemma}, we get the isomorphism $\big(p_*\curly E(-lC)\big)\dual\dual \cong (p_*\curly E)\dual\dual.$
\end{enumerate}\vspace{-.7cm}\end{proof}

\def\cprime{$'$}

\end{document}